\documentclass{amsart}

\newenvironment{thmenum}{

\begin{enumerate}}
{\end{enumerate}}

\usepackage[dvips]{graphicx}
\usepackage{epsfig, latexsym, amsfonts, amsthm, amsmath, pstricks, mathtools}

\DeclareMathOperator{\Hyp}{Hyp}
\DeclareMathOperator{\disc}{disc}
\DeclareMathOperator{\PGO}{PGO}
\DeclareMathOperator{\Sn}{Sn}
\DeclareMathOperator{\End}{End}
\DeclareMathOperator{\ad}{ad}
\DeclareMathOperator{\Orth}{O}
\DeclareMathOperator{\clif}{clif}
\DeclareMathOperator{\charac}{char}
\renewenvironment{proof}[1][]{\vskip-\lastskip\par\vskip6pt plus2pt
minus0pt\par%
\noindent\textit{Proof.}\enspace\ignorespaces}{\hfill$\Box$\par\vskip6pt
plus2pt minus0pt}

\numberwithin{equation}{section}

\newtheorem{theorem}[equation]{Theorem}

\newtheorem{lemma}[equation]{Lemma}

\newtheorem{proposition}[equation]{Proposition}

\newtheorem{corollary}[equation]{Corollary}

\theoremstyle{definition}
\newtheorem{definition}[equation]{Definition}

\newtheorem{remark}[equation]{Remark}

\newtheorem{notation}[equation]{Notation}

\newcommand{\qform}[1]{{\langle{#1}\rangle}}

\begin{document}
\title[The Kneser-Tits Conjecture For $E_{8,2}^{66}$] {The Kneser-Tits
  Conjecture For Groups with Tits-Index {\boldmath \(E_{8,2}^{66}\)}
  Over an Arbitrary Field}

\author{R. Parimala}
\address{Department of Mathematics and Computer Science \\
  Emory University \\
  MSC W401 400 Dowman Drive \\
  Atlanta, GA 30322, USA} \email{parimala@mathcs.emory.edu}
\author{J.-P. Tignol} \address{Institute for Information and
  Communication Technologies,
  Electronics and Applied Mathematics \\
  Universit\'e Catholique de Louvain \\
  Avenue Georges Lema\^\i tre, 4 \\
  1348 Louvain-la-Neuve, Belgium}
\email{Jean-Pierre.Tignol@uclouvain.be} \author{R. M. Weiss}
\address{Department of Mathematics \\
  Tufts University \\
  503 Boston Avenue \\
  Medford, MA 02155, USA} \email{rweiss@tufts.edu}


\date{\today}

\begin{abstract}
  We prove: (1) The group of multipliers of similitudes of a
  $12$-dimensional anisotropic quadratic form over a field $K$ with
  trivial discriminant and split Clifford invariant is generated by
  norms from quadratic extensions $E/K$ such that $q_E$ is hyperbolic.
  (2) If $G$ is the group of $K$-rational points of an absolutely
  simple algebraic group whose Tits index is $E_{8,2}^{66}$, then $G$
  is generated by its root groups, as predicted by the Kneser-Tits
  conjecture.
\end{abstract}

\maketitle

\section{Introduction}\label{abc0}

\noindent
The Kneser-Tits conjecture---first formulated in
\cite{kneser-tits}---predicts that the group of $K$-rational points
(for some field $K$ of arbitrary characteristic) of an absolutely
simple algebraic group with Tits index
$$\centerline{\epsfbox{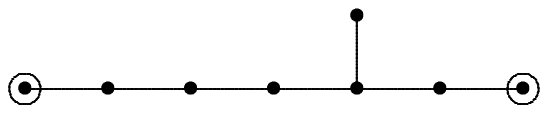}}$$ 
is generated by its root groups. This Tits-index is denoted by
$E_{8,2}^{66}$ in \cite{tits-algebraic}. Groups with this Tits index
are classified by similarity classes of anisotropic $12$-dimensional
quadratic forms over $K$ with trivial discriminant and split Clifford
invariant.  By \cite[42.6]{TW}, they are also the groups whose
corresponding spherical building is a Moufang quadrangle of type $E_8$
as defined in \cite[16.6]{TW}.

Given a quadratic form $q$ defined over a field $K$, we denote by
${\rm clif}(q)$ the Clifford invariant of $q$, by $G(q)$ the group of
multipliers of similitudes of $q$, by ${\rm Hyp}(q)$ the subgroup of
$K^\times$ generated by $K^{\times2}$ and the norms from finite
extensions $E/K$ such that $q_E$ is hyperbolic and by ${\rm Hyp}_2(q)$
the subgroup of ${\rm Hyp}(q)$ generated by $K^{\times2}$ and the
norms from \emph{quadratic} extensions $E/K$ such that $q_E$ is
hyperbolic (including inseparable ones).

Our goal is to prove the following closely related statements.

\begin{theorem}\label{abc2}
  If $q$ is an anisotropic quadratic form with trivial discriminant,
  then $G(q)={\rm Hyp}_2(q)$ in the following cases:
  \begin{thmenum}
  \item $\dim q=8$ and the index of ${\rm clif}(q)$ is~$2$;
  \item $\dim q=12$ and ${\rm clif}(q)$ is split.
  \end{thmenum}
\end{theorem}

\begin{theorem}\label{abc3}
  If $G$ is the group of $K$-rational points of an absolutely simple
  algebraic group whose Tits index is $E_{8,2}^{66}$, then $G$ is
  generated by its root groups.
\end{theorem}

We give two very different proofs of these theorems. In \S\ref{abc97}
we lay the groundwork that is common to the two proofs, and show that
the equality $G(q)=\Hyp(q)$ holds for quadratic forms as in
Theorem~\ref{abc2}. As a consequence, the connected component of the
identity $\PGO_+(q)$ in the group of projective similitudes is
$R$-trivial if $\charac(K)\neq2$: see Corollary~\ref{pdq2}. In
\S\ref{abc4}, we give proofs of Theorems~\ref{abc2} and~\ref{abc3}
based on results in \cite{weiss-quad} and \cite{weiss-crelle}. In
particular, the notion of a {\it quadrangular algebra} introduced in
Chapters~12--13 of \cite{TW} and in \cite{weiss-quad} plays a central
role in these proofs. In \S\ref{abc96} we show how the $R$-triviality
of $\PGO_+(q)$ for $q$ as in Theorem~\ref{abc2}(ii) in
characteristic~$0$ yields another proof of Theorem~\ref{abc3} in
arbitrary characteristic.  In \S\ref{sec:trial} we give an entirely
different proof of Theorem~\ref{abc2} under the assumption that 
$\charac(K)\ne2$, using the triality-defined 
correspondence between $8$-dimensional quadratic forms of trivial
discriminant and hermitian forms over the simple components of their
even Clifford algebra.

For a survey of what is known about the Kneser-Tits conjecture; see
\cite{gille}. We call attention especially to \S6 of that paper, where
the Kneser-Tits conjecture over an arbitrary field is discussed.  By
\cite{prasad}, \cite[6.1]{gille} and Theorem~\ref{abc3}, the only
exceptional groups of relative rank at least~2 for which the
Kneser-Tits conjecture remains to be verified over arbitrary fields
are those whose Tits index is
$$\centerline{\epsfbox{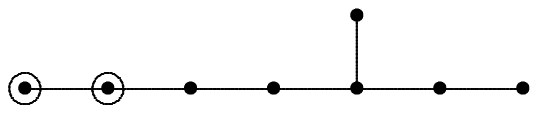}}$$ 
(called $E_{8,2}^{78}$ in \cite{tits-algebraic}). Groups with this
Tits index are classified by isotopy classes of Albert division
algebras, and the corresponding spherical buildings are the Moufang
hexagons defined in \cite[16.8]{TW} for ``hexagonal systems'' of
dimension~$27$. See also \cite[8.6]{gille} and \cite[37.41]{TW}.

\medskip
\noindent {\sc Acknowledgement.}
The proof in \S\ref{abc96} that the $R$-triviality in
characteristic~$0$ of $\PGO_+(q)$ for $q$ as in Theorem~\ref{abc2}(ii)
implies Theorem~\ref{abc3} in arbitrary characteristic is due to Skip
Garibaldi.  We would like to thank him for allowing us to reproduce
his proof here.

\medskip
\section{Similitudes of quadratic forms}\label{abc97}

\noindent 
Our main background reference for quadratic forms is \cite{elman},
although we mostly use the notation of \cite{weiss-quad}.  Let
$(K,L,q)$ be a quadratic space. Thus $K$ is a field, $L$ is a
$K$-vector space and $q\colon L\to K$ is a quadratic form on $L$. We
let $f=\partial q$ denote the polar bilinear form of $q$. Thus
\[
f(x,y)=q(x+y)-q(x)-q(y)\qquad\text{for $x$, $y\in L$.}
\]
The quadratic space $(K,L,q)$ is \emph{nondegenerate} if $\dim_K{\rm
  rad}\, f\le1$; see \cite[7.17]{elman}. If $\dim_KL$ is even and
$(K,L,q)$ is nondegenerate, then $f$ is nondegenerate, the
discriminant $\disc(q)$ is the isomorphism class of the center of the
even Clifford algebra $C_0(q)$ and the Clifford invariant $\clif(q)$
is the Brauer class of the full Clifford algebra $C(q)$; see
\cite[\S\S13, 14]{elman}. As in \cite[\S\S8, 9]{elman}, we let $I_qK$
denote the quadratic Witt group of $K$ and let $I^n_qK=I^{n-1}K\cdot
I_qK$ for all $n>0$, where $I^{n-1}K$ is the $(n-1)$st power of the
fundamental ideal $IK$ of even-dimensional forms in the bilinear Witt
ring $WK$.

The following definitions are taken from \cite[21.31]{TW}.

\begin{definition}\label{abc5}
  A quadratic space $(K,L,q)$ is {\it of type $E_7$} if it is
  anisotropic and there exists a separable quadratic extension $E/K$
  with norm $N$ and scalars $\alpha_1,\ldots,\alpha_4$ such that
$$(K,L,q)\cong(K,E^4,\alpha_1N\perp\alpha_2N\perp\alpha_3N\perp\alpha_4N)$$
and
$$\alpha_1\alpha_2\alpha_3\alpha_4\not\in N(E).$$
In other words, $q$ is anisotropic,
\[
q\cong\qform{\alpha_1,\alpha_2,\alpha_3,\alpha_4}\cdot N,
\]
and the quaternion algebra $(E/K,\alpha_1\alpha_2\alpha_3\alpha_4)$,
which represents $\clif(q)$, is not split.
\end{definition}

\begin{definition}\label{abc5a}
  A quadratic space $(K,L,q)$ is {\it of type $E_8$} if it is
  anisotropic and there exists a separable quadratic extension $E/K$
  with norm $N$ and scalars $\alpha_1,\ldots,\alpha_6$ such that
$$(K,L,q)\cong(K,E^6,\alpha_1N\perp\alpha_2N\perp\cdots\perp\alpha_6N)$$
and
$$-\alpha_1\alpha_2\cdots\alpha_6\in N(E).$$
In other words, $q$ is anisotropic,
\[
q\cong\qform{\alpha_1,\ldots,\alpha_6}\cdot N,
\]
and $\clif(q)$ is split.
\end{definition}

\begin{proposition}\label{abc6}
  Suppose that $(K,L,q)$ is an anisotropic quadratic space. Then the
  following hold:
  \begin{thmenum}
  \item $(K,L,q)$ is of type $E_7$ if and only if $\dim q=8$, $\disc(
    q)$ is trivial and ${\rm clif}(q)$ is of index~2.
  \item $(K,L,q)$ is of type $E_8$ if and only if $\dim q=12$, $\disc(
    q)$ is trivial and ${\rm clif}(q)$ is split. These conditions are
    also equivalent to $\dim q=12$ and $q\in I^3_qK$.
  \end{thmenum}
\end{proposition}

\begin{proof}
  If $\charac(K)\neq2$, (i) is in \cite[Ex.~9.12]{knebusch} and (ii)
  in \cite[p.~123]{pfister}. In arbitrary characteristic, see
  \cite[4.12]{demedts} and (for the second part of (ii))
  \cite[Thm.~16.3]{elman} if $\charac(K)=2$.
\end{proof}

\begin{remark}\label{abc88}
  Suppose that $(K,L,q)$ is a quadratic space of type $E_7$ and that
  $q(1)=1$ for a distinguished element $1$ of $L$. Let $E$ and
  $\alpha_1,\ldots,\alpha_4$ be as in \ref{abc5}. By
  \cite[2.24]{weiss-quad},
  \begin{equation}\label{abc88a}
    C(q,1)\cong M(4,D)\oplus M(4,D),
  \end{equation}
  where $C(q,1)$ is the Clifford algebra with base point as defined in
  \cite[12.47]{TW} and $D$ is the quaternion division algebra
  $(E/K,\alpha_1\alpha_2\alpha_3\alpha_4)$. By \cite[12.51]{TW},
  $C(q,1)$ is isomorphic to the even Clifford algebra $C_0(q)$.  Since
  $D$ represents ${\rm clif}(q)$, it is independent of the choice of
  the orthogonal decomposition of $q$ in \ref{abc5}.
\end{remark}

Our goal in this section is to prove the equality $G(q)=\Hyp(q)$ for
$q$ of type $E_7$ or $E_8$. We start with some general observations.
The following is essentially \cite[2.18]{weiss-quad}.

\begin{proposition}\label{abc14a}
  Suppose that the polynomial $p(x)=x^2-\alpha x+\beta\in K[x]$ is
  separable and irreducible over $K$. Let $E$ be the splitting field
  of $p(x)$ over $K$ and let $N$ denote the norm of the extension
  $E/K$, so that $(K,E,N)$ is a nondegenerate anisotropic
  $2$-dimensional quadratic space. Let $(K,L,q)$ be a
  finite-dimensional quadratic space. Then the following assertions
  are equivalent:
  \begin{thmenum}
  \item The $K$-vector space structure on $L$ extends to an $E$-vector
    space structure such that $q(u\cdot v)=N(u) q(v)$ for all $u\in
    E$, $v\in L$;
  \item There exists a similitude $T$ of $q$ such that $q(T(v))=\beta
    q(v)$ and $f(v,T(v))=\alpha$ for all non-zero $v\in L$ and
    $p(T)=0$;
  \item For each $v_1\in L$ there exists a decomposition
    $L=V_1\oplus\cdots\oplus V_d$ for some $d\in\mathbb{N}$ such that
    $v_1\in V_1$, the restriction $q_i$ of $q$ to $V_i$ is similar to
    $N$ for each $i\in[1,d]$ and $q=q_1\perp\cdots\perp q_d$;
  \item for some $d\in\mathbb{N}$ and some
    $\alpha_1,\alpha_2,\ldots,\alpha_d\in K^\times$,
$$(K,L,q)\cong(K,E^d,\alpha_1N\perp\alpha_2N\perp\cdots\perp\alpha_dN);$$
\item $q_E$ is hyperbolic.
\end{thmenum}
\end{proposition}

\begin{proof}
  Suppose that (i) holds, choose a root $\gamma\in E$ of $p(x)$ and
  let $T(v)=\gamma\cdot v$ for all $v\in L$. Then $T$ is a similitude
  of $q$ as in (ii). If $T$ is a similitude of $q$ as in (ii), then
  for each nonzero $v\in L$ the restriction of $q$ to $\langle
  v,T(v)\rangle$ is similar to $N$ (and, in particular, is
  nondegenerate). Therefore (iii) holds, and (iii) of course implies
  (iv).  Fixing an isomorphism as in (iv), we may transfer to $L$ the
  natural $E$-vector space structure on $E^d$ to obtain~(i).  The
  equivalence of (iv) and (v) follows readily from
  \cite[Prop.~34.8]{elman}.
\end{proof}

\begin{definition}\label{abc70}
  A similitude $\varphi$ of a quadratic space $(K,L,q)$ is called
  \emph{inseparable} if $\charac(K)=2$, the multiplier of $\varphi$ is
  not in $K^{\times2}$ and
$$f(v,\varphi(v))=0\qquad\text{for all $v\in L$},$$ 
where $f=\partial q$. We call a similitude of $q$ {\it separable} if
it is not inseparable. Thus, if $\charac(K)\neq2$ all similitudes are
separable.
\end{definition}

\begin{proposition}\label{abc14p}
  Let $(K,L,q)$ be a finite-dimensional quadratic space such that
  $f=\partial q$ is nondegenerate. If $(K,L,q)$ admits an inseparable
  similitude with multiplier $\gamma$, then
  \[
  q\simeq\qform{1,\gamma}\cdot q_0
  \]
  for some non-degenerate quadratic form $q_0$. In particular, $\dim
  L\equiv0\bmod 4$ and $q_{K(\sqrt{\gamma})}$ is hyperbolic.
\end{proposition}

\begin{proof}
  Let $E=K(\sqrt{\gamma})$ be a purely inseparable quadratic extension
  of $K$, and let $\varphi$ be an inseparable similitude of $(K,L,q)$
  with multiplier $\gamma$. Linearizing the condition
  $f\bigl(v,\varphi(v)\bigr)=0$, we obtain
  \[
  f\bigl(v,\varphi(w)\bigr) = f(\varphi(v),w) \qquad\text{for all $v$,
    $w\in L$.}
  \]
  Since $f\bigl(\varphi(v),\varphi(w)\bigr)=\gamma f(v,w)$ for all
  $v$, $w\in L$, it follows that
  \[
  f\bigl(v,\varphi^2(w)\bigr)= f\bigl(\varphi(v),\varphi(w)\bigr) =
  \gamma f(v,w)\qquad\text{for all $v$, $w\in L$,}
  \]
  hence $\varphi^2(w)=\gamma w$ for all $w\in L$. We then define on
  $L$ an $E$-vector space structure by
  \[
  (\lambda+\mu\sqrt{\gamma})\cdot v = \lambda v +\mu\varphi(v)
  \qquad\text{for $\lambda$, $\mu\in K$ and $v\in L$,}
  \]
  and we define a map $f'\colon L\times L\to E$ by
  \[
  f'(v,w)=f(v,w)+\sqrt{\gamma^{-1}}f\bigl(v,\varphi(w)\bigr)
  \qquad\text{for $v$, $w\in L$.}
  \]
  A straightforward computation shows that $f'$ is a bilinear
  alternating form on $L$. It is nondegenerate since $f$ is
  nondegenerate. Therefore, the dimension of $L$ over $E$ is even,
  hence its dimension over $K$ is a multiple of~$4$. Let
  $(e_i,e'_i)_{i=1}^d$ be a symplectic $E$-base of $L$ for $f'$. If
  $L_0\subset L$ is the $K$-span of $(e_i,e'_i)_{i=1}^d$ and $q_0$ is
  the restriction of $q$ to $L_0$, we have $L=L_0\perp\varphi(L_0)$
  and $q=q_0\perp\qform{\gamma}q_0$.
\end{proof}

\begin{corollary}\label{cor:insepsim}
  Let $(K,L,q)$ be a finite-dimensional quadratic space such that
  $\partial q$ is non-degenerate. Then the multiplier of every
  inseparable similitude of $q$ is in $\Hyp_2(q)$.

\end{corollary}

\begin{proof}
  Let $\gamma$ be the multiplier of an inseparable similitude of $q$.
  Clearly, $\gamma\in N\bigl(K(\sqrt{\gamma})\bigr)$, and
  Proposition~\ref{abc14p} shows that $q$ is hyperbolic over
  $K(\sqrt{\gamma})$.
\end{proof}

We now consider quadratic forms of low dimension. The following result
is presumably well-known:

\begin{proposition}
  \label{prop:isotlowdim}
  Every $10$-dimensional quadratic form in $I^3_q(K)$ is isotropic.
\end{proposition}

\begin{proof}
  This was proved by Pfister \cite[p.~123]{pfister} under the
  hypothesis that $\charac(K)\neq2$. The arguments also apply when
  $\charac(K)=2$; see \cite[Thm.~4.10]{demedts}.
\end{proof}

We now consider quadratic spaces of type~$E_7$. For the next
statement, we do not require the form to be anisotropic.

\begin{lemma}
  \label{lem:dim8}
  Let $(K,L,q)$ be a nondegenerate quadratic space of
  dimension~$8$. If $\disc(q)$ is trivial and $\clif(q)$ is
  represented by a quaternion algebra $Q$ with norm form $N_Q$, then
  $q$ is Witt-equivalent to the sum of a multiple of $N_Q$ and a
  multiple of some $3$-fold Pfister quadratic form $\pi$: there exist
  $\alpha$, $\beta\in K^\times$ such that
  \begin{equation}
    \label{eq:dim8}
    q=\qform{\alpha}\cdot N_Q + \qform{\beta}\cdot \pi \qquad\text{in
      $I_qK$.} 
  \end{equation}
  Moreover, $G(q)=G(N_Q)\cap G(\pi)$.
\end{lemma}

\begin{proof}
  Let $\alpha\in K^\times$ be a value represented by $q$. Consider the
  form
  \[
  q'=q\perp\qform{-\alpha}\cdot N_Q.
  \]
  This $12$-dimensional form is isotropic and has trivial discriminant
  and Clifford invariant, hence it is in $I^3_qK$ and is
  Witt-equivalent to a $10$-dimensional form. By
  Proposition~\ref{prop:isotlowdim}, it is actually equivalent to an
  $8$-dimensional form. By the Arason--Pfister Hauptsatz
  \cite[Thm.~23.7]{elman}, this $8$-dimensional quadratic form becomes
  hyperbolic over the function field of the corresponding quadric,
  hence it is a multiple of some $3$-fold Pfister quadratic form $\pi$
  by \cite[Cor.~23.4]{elman}. Letting $q'=\qform{\beta}\cdot\pi$ in
  $I_qK$, we have \eqref{eq:dim8}.

  Now, for $\gamma\in G(q)$ we have $\qform{1,-\gamma}\cdot q=0$ in
  $I_qK$, hence
  \[
  \qform{1,-\gamma}\cdot\qform{\alpha}\cdot N_Q =
  -\qform{1,-\gamma}\cdot\qform{\beta}\cdot\pi\qquad\text{in $I_qK$}.
  \]
  Since the left side is a form of dimension $8$ and the right side is
  a form of dimension~$16$, the right side must be isotropic. It is
  then hyperbolic by \cite[Cor.~9.10]{elman}, since it is a multiple
  of a Pfister form. The left side is then also hyperbolic, which
  means that $\gamma$ is in $G(\pi)$ and in $G(N_Q)$. We have thus
  proved $G(q)\subset G(N_Q)\cap G(\pi)$. Since the reverse inclusion
  is clear, the proof is complete.
\end{proof}

Lemmas~\ref{lem:biquad} and~\ref{lem:pfister} are well-known when
$\charac(K)\neq2$; see \cite[2.13]{elmanlam} for
Lemma~\ref{lem:biquad}. The proofs we give below
do not require any separability hypothesis.

\begin{lemma}
  \label{lem:biquad}
  Let $E_1$, $E_2$ be linearly disjoint quadratic extensions of a
  field $K$, and let $M=E_1\otimes_KE_2$. The norm groups of $E_1$,
  $E_2$, and $M$ are related as follows:
  \[
  N(E_1/K)\cap N(E_2/K) = K^{\times2}\cdot N(M/K).
  \]
\end{lemma}

\begin{proof}
  Since $K^{\times2}\subset N(E_i/K)$ and $N(M/K)\subset N(E_i/K)$ for
  $i=1$, $2$, the inclusion $N(E_1/K)\cap N(E_2/K) \supset
  K^{\times2}\cdot N(M/K)$ is clear, and it suffices to prove the
  reverse inclusion. We identify $E_1$ and $E_2$ with subfields of $M$
  and consider $\alpha\in N(E_1/K)\cap N(E_2/K)$. Let $x_1\in
  E_1^\times$, $x_2\in E_2^\times$ be such that
  \[
  \alpha=N_{E_1/K}(x_1) = N_{E_2/K}(x_2).
  \]
  Let $T_{E_i/K}\colon E_i\to K$ be the trace map, for $i=1$,
  $2$. Computation shows that
  \[
  x_1N_{M/E_1}(1+x_1^{-1}x_2) = T_{E_1/K}(x_1) + T_{E_2/K}(x_2).
  \]
  If $x_1\neq-x_2$, the left side is nonzero. Taking the norm from
  $E_1$ to $K$ of each side yields
  \[
  \alpha N_{M/K}(1+x_1^{-1}x_2) = \bigl(T_{E_1/K}(x_1) +
  T_{E_2/K}(x_2)\bigr)^2\in K^{\times2},
  \]
  hence $\alpha\in K^{\times2}\cdot N(M/K)$.  If $x_1=-x_2$, then
  $x_1\in E_1\cap E_2=K$, hence $\alpha\in K^{\times2}$.
\end{proof}

\begin{lemma}
  \label{lem:pfister}
  Any multiplier of similitude of an anisotropic quadratic Pfister
  space $(K,L,\pi)$ is a square in $K$ or is the norm of a quadratic
  extension over which $\pi$ is hyperbolic.
\end{lemma}

\begin{proof}
  By \cite[Cor.~9.9]{elman}, the multipliers of $\pi$ are the
  represented values of $\pi$, so any $\gamma\in G(\pi)$ has the form
  $\gamma=\pi(v)$ for some $v\in L$. Let $e\in L$ be such that
  $\pi(e)=1$. If $e$ and $v$ are not linearly independent, then
  $\gamma\in K^{\times2}$. Otherwise, let $V$ be the $K$-span of $e$
  and $v$. The restriction of $\pi$ to $V$ is the norm form of a
  quadratic extension of $K$ over which $\pi$ is isotropic, hence
  hyperbolic by \cite[Cor.~9.10]{elman}. By construction, this norm
  form represents $\gamma$.
\end{proof}

\begin{proposition}
  \label{prop:E7G=H}
  For any nondegenerate quadratic space $(K,L,q)$ of dimension~$8$
  such that $\disc(q)$ is trivial and $\clif(q)$ has index~$1$ or $2$,
  we have $G(q)=\Hyp(q)$.
\end{proposition}

\begin{proof}
  It suffices to show $G(q)\subset\Hyp(q)$, since the reverse
  inclusion follows from the similarity norm principle
  \cite[Thm.~20.14]{elman}. Let $\gamma\in G(q)$, and consider a
  decomposition of $q$ as in \eqref{eq:dim8}. We may assume $q$ is not
  hyperbolic, otherwise $\Hyp(q)=K^\times=G(q)$ and there is nothing
  to prove. If $N_Q$ or $\pi$ is isotropic, hence hyperbolic, the
  proposition readily follows from Lemma~\ref{lem:pfister}. For the
  rest of the proof, we may thus assume $N_Q$ and $\pi$ are
  anisotropic.  By Lemma~\ref{lem:dim8} we have $\gamma\in G(N_Q)\cap
  G(\pi)$, hence Lemma~\ref{lem:pfister} yields quadratic extensions
  $E_1$, $E_2$ of $K$ that split $N_Q$ and $\pi$ respectively, such
  that $\gamma\in N(E_1/K)\cap N(E_2/K)$. If $E_1\cong E_2$, then
  $E_1$ splits $N_Q$ and $\pi$, hence also $q$. Since $\gamma\in
  N(E_1/K)$, it follows that $\gamma\in \Hyp(q)$. If $E_1\not\cong
  E_2$, then $E_1$ and $E_2$ are linearly disjoint over $K$, and the
  tensor product $M=E_1\otimes_KE_2$ is a field that splits $N_Q$ and
  $\pi$, hence also $q$. Since $\gamma$ is a norm from $E_1$ and from
  $E_2$, Lemma~\ref{lem:biquad} shows that $\gamma\in K^{\times2}\cdot
  N(M/K)$, hence $\gamma\in\Hyp(q)$.
\end{proof}

Proposition~\ref{prop:E7G=H} applies in particular to quadratic spaces
of type $E_7$. We now turn to spaces of type $E_8$.

\begin{proposition}\label{pdq1}
  Suppose that $(K,L,q)$ is of type $E_8$, that $\gamma$ is the
  multiplier of a separable similitude of $q$ and that $\gamma\not\in
  K^{\times2}$. Then there exists a decomposition
$$(K,L,q)=(K,L_1,q_1)\perp(K,L_2,q_2)$$
such that $q_2$ is of type $E_7$, $q_1$ is similar to the reduced norm
of the quaternion division algebra representing ${\rm clif}(q_2)$ and
$\gamma$ is a multiplier of similitudes of both $q_1$ and $q_2$.
\end{proposition}

\begin{proof}
  Let $\varphi$ be a separable similitude of $q$ with multiplier
  $\gamma$.  If ${\rm char}(K)=2$, we choose $v\in L$ such that
  $f(v,\varphi(v))\ne0$; if ${\rm char}(K)\ne2$, we let $v$ be an
  arbitrary non-zero vector in $L$.  Next we set $W=\langle
  v,\varphi(v)\rangle$.  Since $\gamma\not\in K^{\times2}$, we have
  $\dim_KW=2$.  Let $\hat q_1$ denote the restriction of $q$ to $W$
  and let $\hat q_2$ denote the restriction of $q$ to $W^\perp$.  The
  form $\hat q_1$ is similar to the norm $N$ of a quadratic extension
  $E/K$ such that $\gamma\in N(E)$.  Since $\hat q_1$ is
  nondegenerate, the extension $E/K$ is separable and $q=\hat
  q_1\perp\hat q_2$. Since $q_E$ and $(\hat q_1)_E$ have trivial
  discriminant and split Clifford invariant, also $(\hat q_2)_E$ has
  trivial discriminant and split Clifford invariant. By
  Proposition~\ref{prop:isotlowdim}, it follows that $(\hat q_2)_E$ is
  isotropic. Hence we can choose a 2-dimensional subspace $U$ of
  $W^\perp$ such that the restriction of $\hat q_2$ to $U$ is
  hyperbolic over $E$.  By Proposition~\ref{abc14a}, the restriction
  of $\hat q_2$ to $U$ is similar to $N$. Let $L_1=W\oplus U$, let
  $L_2=L_1^\perp$ and let $q_i$ denote the restriction of $q$ to $L_i$
  for $i=1$ and $2$. Then $q_1$ is similar to the reduced norm of a
  quaternion division algebra $D$, $\gamma$ is a multiplier of $q_1$
  and
$$(K,L,q)=(K,L_1,q_1)\perp(K,L_2,q_2).$$
Since $\gamma$ is a multiplier of both $q$ and $q_1$, both $\langle
1,-\gamma\rangle\cdot q$ and $\langle 1,-\gamma\rangle\cdot q_1$ are
hyperbolic. Since
$$\langle 1,-\gamma\rangle\cdot q=\langle 1,-\gamma\rangle\cdot 
q_1\perp\langle 1,-\gamma\rangle\cdot q_2,$$ it follows by Witt's
Cancellation Theorem that the product $\langle 1,-\gamma\rangle\cdot
q_2$ is also hyperbolic. Hence $q_2=\qform{\gamma}\cdot q_2$ in
$I_qK$. By \cite[8.17]{elman}, therefore, there is a similitude of
$q_2$ with multiplier $\gamma$.  Since $\disc(q)$ and $\disc(q_1)$ are
both trivial, so is $\disc(q_2)$.  Furthermore, ${\rm clif}(q_2)={\rm
  clif}(q_1)$ since ${\rm clif}(q)$ is split.  Since the quaternion
division algebra $D$ represents ${\rm clif}(q_1)$, it also represents
${\rm clif}(q_2)$. We conclude, in particular, that $q_2$ is of type
$E_7$.
\end{proof}

\begin{corollary}
  \label{cor:E8G=H}
  For any nondegenerate quadratic space $(K,L,q)$ of dimension~$12$
  such that $\disc(q)$ and $\clif(q)$ are trivial, we have
  $G(q)=\Hyp(q)$.
\end{corollary}

\begin{proof}
  As in Proposition~\ref{prop:E7G=H}, it suffices to prove
  $G(q)\subset\Hyp(q)$. If $q$ is isotropic, then
  Proposition~\ref{prop:isotlowdim} shows that $q$ is Witt-equivalent
  to a multiple of a $3$-fold Pfister form, hence the inclusion
  follows from Lemma~\ref{lem:pfister}. For the rest of the proof, we
  may thus assume $q$ is anisotropic, i.e., $q$ is of type~$E_8$.

  Let $\gamma\in G(q)$. If $\gamma$ is the multiplier of an
  inseparable similitude, then we have $\gamma\in\Hyp(q)$ by
  Corollary~\ref{cor:insepsim}. If $\gamma$ is the multiplier of a
  separable similitude, we fix a decomposition $q=q_1\perp q_2$ as in
  Proposition~\ref{pdq1}, so $\gamma\in G(q_1)\cap G(q_2)$. By
  Proposition~\ref{prop:E7G=H} we have $G(q_2)=\Hyp(q_2)$. Now, if
  $E/K$ is a finite extension such that $(q_2)_E$ is hyperbolic, then
  $E$ splits $\clif(q_2)$. Hence $(q_1)_E$ is hyperbolic, and
  therefore $q_E$ is hyperbolic. This shows
  $\Hyp(q_2)\subset\Hyp(q)$. Since $\gamma\in \Hyp(q_2)$, it follows
  that $\gamma\in \Hyp(q)$.
\end{proof}

\begin{remark}
  \label{rem:E8G=H}
  Restricting to \emph{quadratic} extensions that split $q_2$ in the
  last part of the proof above, we see that
  $\Hyp_2(q_2)\subset\Hyp_2(q)$. This observation will be used in
  \S\ref{sec:trial}.
\end{remark}

When $\charac(K)\neq2$, Proposition~\ref{prop:E7G=H} and
Corollary~\ref{cor:E8G=H} yield information on the connected component
of the identity $\PGO_+(q)$ in the group of projective similitudes of
$(K,L,q)$, which is the group of algebra automorphisms of
$\operatorname{End}(L)$ that commute with the adjoint involution of
$q$. The property of $R$-triviality used in the following statement
refers to Manin's $R$-equivalence; see \cite[\S1]{merkurjev1} for
details.

\begin{corollary}
  \label{pdq2}
  Assume $\charac(K)\neq2$. For $q$ a quadratic form of type $E_7$ or
  $E_8$ over $K$, the group $\PGO_+(q)$ is $R$-trivial.
\end{corollary}

\begin{proof}
  By \cite[Thm.~1]{merkurjev1}, it suffices to prove that
  $G(q_E)=\Hyp(q_E)$ for every field $E$ containing $K$. If $q$ is of
  type $E_7$, this property readily follows from
  Proposition~\ref{prop:E7G=H}. If it is of type $E_8$, it follows
  from Corollary~\ref{cor:E8G=H}.
\end{proof}

\begin{remark}
  Skip Garibaldi has observed that if $q$ is of type $E_7$, then it
  follows from Lemma~\ref{lem:dim8} and \cite[Prop.~6.1]{garibaldi}
  that the group $\PGO_+(q)$ is actually stably rational.
\end{remark}

\smallskip
\section{Quadrangular algebras and proofs of Theorems~\ref{abc2}
  and~\ref{abc3}}\label{abc4}

\noindent
Most of this section is devoted to results about quadrangular
algebras. At the very end of this section, we use these results to
prove Theorems~\ref{abc2} and~\ref{abc3}.

The notion of a quadrangular algebra arose in the course of the
classification of Moufang polygons; see, in particular,
Chapters~12--13 and~27 in \cite{TW}.  For the definition, see
\cite[1.17]{weiss-quad}.

\begin{proposition}\label{abc7}
  Let $(K,L,q)$ be a quadratic space of type $E_7$ or $E_8$ as defined
  in \ref{abc5} and~\ref{abc5a} and suppose that $q(1)=1$ for a
  distinguished element $1$ of $L$. Then there exists a unique
  quadrangular algebra
$$\Xi=(K,L,q,1,X,\cdot,h,\theta)$$
as defined in \cite[1.17]{weiss-quad}.
\end{proposition}

\begin{proof}
  Existence holds by \cite[Thm.~10.1]{weiss-quad} and uniqueness (up
  to equivalence as defined in \cite[Thm.~1.22]{weiss-quad}) holds by
  \cite[6.42]{weiss-quad}.
\end{proof}

\begin{notation}\label{abc7x}
  For the rest of this section, we let
$$\Xi=(K,L,q,1,X,\cdot,h,\theta)$$
be as in \ref{abc7} and $f=\partial q$. By
\cite[Prop.~4.2]{weiss-quad}, we can assume that $\Xi$ is
$\delta$-standard for some $\delta\in L$ as defined in
\cite[4.1]{weiss-quad}. (This allows us to use the identities in
Chapter~4 of \cite{weiss-quad}.) In addition, we let $\sigma$ be as
\cite[1.2]{weiss-quad}, we let $u^{-1}$ for all non-zero $u\in L$ be
as in \cite[1.3]{weiss-quad} and we let $\pi$ be as in
\cite[1.17(D1)]{weiss-quad}.
\end{notation}

\begin{remark}\label{abc8}
  Suppose that $(K,L,q)$ is of type $E_7$ and let $C(q,1)$ and $D$ be
  as in \ref{abc88}. By \eqref{abc88a} and \cite[1.17(A1)-(A3) and
  Prop.~2.22]{weiss-quad}, there exists a unique map $*$ from $D\times
  X$ to $X$ with respect to which $X$ is a left vector space over $D$,
  $ta=t*a$ for all $(a,t)\in X\times K$ and $w*(a\cdot v)=(w*a)\cdot
  v$ for all $w\in D$, $a\in X$ and $v\in L$.  This map is given
  explicitly in \cite[3.6]{weiss-crelle}.
\end{remark}

\begin{proposition}\label{abc9}
  Suppose that $(K,L,q)$ is of type $E_7$, let $D$ and $*$ be as in
  \ref{abc8}, let $\varphi_1$ be a similitude of $q$, let
  $u=\varphi_1(1)$ and let
$$a\ \hat\cdot\ v=(a\cdot v)\cdot u^{-1}$$
for all $(a,v)\in X\times L$, where $u^{-1}$ is as in \ref{abc7x}.
Then there exists a similitude $\varphi$ with the same multiplier as
$\varphi_1$ such that $u=\varphi(1)$, an element $\omega\in K^\times$
and a $D$-linear automorphism $\psi$ of $X$ such that the following
hold:
\begin{thmenum}
\item $\psi(a\cdot v)=\psi(a)\ \hat\cdot\ \varphi(v)$ for all $a\in X$
  and all $v\in L$.
\item $\varphi(h(a,b))=\omega h(\psi(a),\psi(b)u)$ for all $a,b\in X$.
\item $\varphi(\theta(a,v))\equiv\omega\theta(\psi(a),\varphi(v)) \
  ({\rm mod}\ \langle\varphi(v)\rangle)$ for all $a\in X$ and all
  $v\in L$.
\end{thmenum}
\end{proposition}

\begin{proof}
  The map $\varphi_1$ is an isomorphism of pointed quadratic spaces
  from $(K,L,q,1)$ to $(K,L,q/q(u),u)$. It therefore induces an
  isomorphism of Clifford algebras with base point from $C(q,1)$ to
  $C(q/q(u),u)$. Let $\hat\Xi$, $\hat h$ and $\hat\theta$ be as in
  \cite[Prop.~8.1]{weiss-quad}; thus,
  $\hat\Xi=(K,L,q/q(u),u,X,\,\hat\cdot\,,\hat h,\hat\theta)$ is the
  isotope of $\Xi$ at $u$ as defined in \cite[8.7]{weiss-quad}.  By
  \cite[1.17(A1)--(A3) and Prop.~2.22]{weiss-quad} applied to both
  $\Xi$ and to $\hat\Xi$, $X$ is a right $C(q,1)$-module with respect
  to $\cdot$ and a right $C(q/q(u),u)$-module with respect to
  $\hat\cdot$.  By \cite[12.55]{TW} (where the base point $1$ is
  called $\epsilon$), exactly one of the two direct summands in
  \eqref{abc88a} acts nontrivially on $X$, and by \cite[12.54]{TW},
  there exists an isometry $\rho$ of $q$ fixing $1$ that extends to an
  automorphism of $C(q,1)$ interchanging the two direct summands. Thus
  for $j=0$ or $1$, the composition $\varphi_1\circ\rho^j$ maps the
  direct summand $A$ of $C(q,1)$ acting nontrivially on $X$ to the
  direct summand $A_u$ of $C(q/q(u),u)$ acting nontrivially on
  $X$. Let $\varphi=\varphi_1\circ\rho^j$.  Choosing a basis for $X$
  as a left vector space over $D$, we can identify both $A$ and $A_u$
  with ${\rm End}_D(X)$.  It follows that there exists a $D$-linear
  automorphism $\psi$ of $X$ such that (i) holds. By \cite[1.25 and
  Prop.~6.38]{weiss-quad}, there exists $\omega\in K^\times$ such that
  also (ii) and (iii) hold.
\end{proof}

\begin{notation}\label{abc10}
  Let $g$ and $\phi$ be the maps that appear in
  \cite[1.17(C3)-(C4)]{weiss-quad}, let
$$(U_+,U_1,U_2,U_3,U_4)$$
be the root group sequence and $x_4$ the isomorphism from $L$ to $U_4$
obtained by applying the recipe in \cite[16.6]{TW} to $\Xi$, $g$ and
$\phi$, let $\Gamma$ be the corresponding Moufang quadrangle (see
\cite[8.11]{TW}), let $G^\dagger$ be the subgroup of ${\rm
  Aut}(\Gamma)$ generated by the root groups of $\Gamma$, let $H_0$ be
the subgroup of ${\rm Aut}(\Gamma)$ defined in
\cite[1.4]{weiss-crelle} (or \cite[11.20]{weiss-quad}), let
$G=H_0\cdot G^\dagger$ and let $H^\dagger=H_0\cap G^\dagger$.  By
\cite[35.11]{TW}, the similarity class of $(K,L,q)$ is an invariant of
$\Gamma$.
\end{notation}

\begin{proposition}\label{abc10a}
  $G/G^\dagger\cong H_0/H^\dagger$ and if $(K,L,q)$ is of type $E_8$,
  then $G$ is the group of $K$-rational points of an absolutely simple
  algebraic group with Tits index $E_{8,2}^{66}$, and every such group
  arises in this way starting with some quadratic space of type $E_8$
  defined over $K$.
\end{proposition}

\begin{proof}
  For the isomorphism $G/G^\dagger\cong H_0/H^\dagger$, see the top of
  page~193 of \cite{weiss-crelle} (where $G$ is called $G_0$), The
  remaining assertions hold by \cite[42.6]{TW}.
\end{proof}

\begin{proposition}\label{abc11}
  Suppose that $(K,L,q)$ is of type $E_7$ and that $\varphi_1$ is a
  similitude of $q$. Let $H_0$ and $x_4$ be as in \ref{abc10}. Then
  there exist an element $h$ of $H_0$ and a similitude $\varphi$ of
  $q$ with the same multiplier as $\varphi_1$ such that
$$x_4(v)^h=x_4(\varphi(v))$$ 
for all $v\in L$.
\end{proposition}

\begin{proof}
  Let $\varphi$ and $\psi$ be the maps obtained by applying
  Proposition~\ref{abc9} to $\varphi_1$.  By Proposition~\ref{abc9}
  and \cite[Prop.~12.5]{weiss-quad}, the pair $(\varphi,\psi)$ is
  contained in the structure group of $\Xi$ (as defined in
  \cite[12.4]{weiss-quad}).  The claim follows now by
  \cite[Thm.~12.11]{weiss-quad} and the first few lines of its proof
  (as well as \cite[11.22]{weiss-quad}).
\end{proof}

From now on we identify $K$ with its image under the map $t\mapsto
t\cdot 1$ from $K$ to $L$. Thus when we write $\pi(a)+t$ for $(a,t)\in
X\times K$, for example, we mean $\pi(a)+t\cdot 1$ (where $\pi$ is as
in \ref{abc7x}).

\begin{proposition}\label{abc12}
  Let $a$ be a non-zero element of $X$, let
$$p(x)=x^2-f(1,\pi(a))x+q(\pi(a))\in K[x],$$
let $E$ be the splitting field of $p(x)$ over $K$ and let $N$ be the
norm of the extension $E/K$.  Let $T(v)=\theta(a,v)$ for all $v\in L$
and let $I$ be the identity automorphism of $L$. Then
$N(E^\times)=K^{\times2}\cdot\{q(\pi(a)+t)\mid t\in K\}$, $q_E$ is
hyperbolic and for each $t\in K$, $T+tI$ is a similitude of $q$ with
multiplier $q(\pi(a)+t)$.
\end{proposition}

\begin{proof}
  By \cite[1.17(D2)]{weiss-quad}, $p(t)=q(\pi(a)-t)\ne0$ for each
  $t\in K$.  Thus $p(x)$ is irreducible over $K$. It follows that
  $N(E^\times)=K^{\times2}\cdot\{q(\pi(a)+t)\mid t\in K\}$.  By
  \cite[Props.~4.9(i) and~4.22]{weiss-quad}, $T+tI$ is a similitude of
  $q$ with multiplier $q(\pi(a)+t)$ for each $t\in K$ and
  $f(T(v),v)=f(\pi(a),1)q(v)$ for each non-zero $v\in L$.  By
  \cite[Prop.~4.21]{weiss-quad}, $p(T)=0$. Thus if $p(x)$ is
  separable, then $q_E$ is hyperbolic by Propositions~\ref{abc14a}.
  If $p(x)$ is inseparable, then $f(\pi(a),1)=0$, hence $T$ is
  inseparable and again $q_E$ is hyperbolic, this time by
  Proposition~\ref{abc14p}.
\end{proof}

\begin{definition}\label{abc13c}
  For each non-zero $a\in X$, the map $v\mapsto\theta(a,v)$ is a
  similitude of $q$ by Proposition~\ref{abc12}.  We call an element
  $a\in X$ {\it separable} if $a\ne0$ and the similitude
  $v\mapsto\theta(a,v)$ of $q$ is separable as defined in \ref{abc70}.
  We let $X_{\rm sep}$ denote the set of separable elements of
  $X$. Thus if ${\rm char}(K)\ne2$, then $X_{\rm
    sep}=X\backslash\{0\}$, but if ${\rm char}(K)=2$, then by
  \cite[Prop.~4.9(i)]{weiss-quad},
$$X_{\rm sep}=\{a\in X\mid f(\pi(a),1)\ne0\}.$$
If ${\rm char}(K)=2$, then by \cite[13.42--13.43]{TW}, $a\mapsto
f(\pi(a),1)$ is a nondegenerate quadratic form on $X$.  In particular,
the set $X_{\rm sep}$ is non-empty also if ${\rm char}(K)=2$.
\end{definition}

\begin{proposition}\label{abc13}
  Every inseparable similitude of $q$ (as defined in \ref{abc70}) is
  the product of two separable similitudes.
\end{proposition}

\begin{proof}
  Let $\varphi$ be an inseparable similitude of $q$, so ${\rm
    char}(K)=2$.  It suffices to show that
$$f\big(\theta(a,\varphi(v)),v\big)\ne0$$
for some $v\in L$ and some $a\in X_{\rm sep}$, where $X_{\rm sep}$ is
as in \ref{abc13c}.  Suppose this is false and let
$w=\varphi(1)$. Then
\begin{equation}\label{abc13a}
  f\big(\theta(a,w),1\big)=0
\end{equation}
for all $a\in X_{\rm sep}$. Furthermore,
\begin{equation}\label{abc13b}
  f(w,1)=0
\end{equation}
but $w\not\in\langle1\rangle$ since $\varphi$ is inseparable.  Choose
$a\in X_{\rm sep}$.  Since $f$ is nondegenerate, we can choose
$v\in\langle w\rangle^\perp
\backslash\langle1\rangle^\perp$. Replacing $v$ by $v+w$ if necessary,
we can assume in addition (by \cite[Prop.~4.9(i)]{weiss-quad} again)
that
\begin{equation}\label{abc13d}
  f(\theta(a,w),v)\ne0.
\end{equation}
By \cite[Prop.~3.21]{weiss-quad}, $av\in X_{\rm sep}$, so
$f\big(\theta(av,w),1)\big)=0$ by \eqref{abc13a}. By
\cite[1.17(C4)]{weiss-quad} and \eqref{abc13b}, it follows that
$$f(\theta(a,w^\sigma)^\sigma,1)q(v)=f(w,v^\sigma)f(\theta(a,v)^\sigma,1)
+f(\theta(a,v),w^\sigma)f(v^\sigma,1),$$ where $\sigma$ is as in
\ref{abc7x}. By \cite[1.4]{weiss-quad} and \eqref{abc13b}, we have
$x^\sigma=x$ for $x=1$ and $x=w$ and $f(x^\sigma,y)=f(x,y^\sigma)$ for
all $x,y\in L$.  Therefore
$$f(\theta(a,w^\sigma)^\sigma,1)=f(\theta(a,w),1)=0$$
by \eqref{abc13a} and
$$f(w,v^\sigma)=f(w^\sigma,v)=f(w,v)=0$$
by the choice of $v$ and hence
$$f(\theta(a,v),w)f(v,1)=f(\theta(a,v),w^\sigma)f(v^\sigma,1)=0.$$
Since $f(v,1)\ne0$ by the choice of $v$, we conclude that
$f(\theta(a,v),w)=0$. By \cite[Prop.~4.22]{weiss-quad}, therefore,
$f\big(\theta(a,\theta(a,v)),\theta(a,w)\big)=0$. By
\cite[Prop.~4.21]{weiss-quad}, it follows that
$$f(\pi(a),1)f\big(\theta(a,v),\theta(a,w)\big)=q(\pi(a))f\big(v,\theta(a,w)\big).$$
By \eqref{abc13d}, therefore,
$f\big(\theta(a,v),\theta(a,w)\big)\ne0$. By one more application of
\cite[Prop.~4.22]{weiss-quad}, however,
$f\big(\theta(a,v),\theta(a,w)\big)=q(\pi(a))f(v,w)=0$.
\end{proof}

\begin{proposition}\label{abc14b}
  Let $E/K$ be a separable quadratic extension such that $q_E$ is
  hyperbolic and let $V_i$ and $q_i$ for $i\in[1,d]$ be as in
  Proposition~\ref{abc14a}(iii) with $v_1=1$.  Then there exists $e\in
  X_{\rm sep}$ such that $\theta(e,V_i)=V_i$ for each $i\in[1,d]$.
\end{proposition}

\begin{proof}
  Let $p(x)=x^2-\alpha x+\beta\in K[x]$ be an irreducible polynomial
  that splits over $E$.  We can choose $p(x)$ so that $\alpha=0$ if
  and only if ${\rm char}(K)\ne2$.  Let $\gamma,\gamma_1\in E$ be the
  two roots of $p(x)$.  There exists an $E$-vector space structure on
  $L$ as in Proposition~\ref{abc14a}(i) such that $V_i$ is a
  1-dimensional subspace for each $i\in[1,d]$.  Let $T(v)=\gamma\cdot
  v$ for each $v\in L$ and let $T^\epsilon$ be the unique automorphism
  of $L$ such that $T^\epsilon(v)=\gamma_1\cdot v$ for all $v\in V_1$
  and $T^\epsilon(v)=T(v)$ for all $v\in V_1^\perp$. Both $T$ and
  $T^\epsilon$ are norm splitting maps of $q$ as defined in
  \cite[12.14]{TW} and both map $V_i$ to itself for each $i\in[1,d]$.
  By \cite[12.20 and 13.13(ii)]{TW}, therefore, we can choose
  $R\in\{T,T^\epsilon\}$ such that $R$ is linked to the map
  $(a,v)\mapsto a\cdot v$ at some point $e\in X$ as defined in
  \cite[13.2]{TW}. By \cite[13.61]{TW}, $e\in X_{\rm sep}$ and there
  exists $r\in K^\times$ and $s\in K$ such that $R(v)=r\theta(e,v)+sv$
  for all $v\in L$.
\end{proof}

\begin{proposition}\label{abc14}
  Suppose that $(K,L,q)$ is of type $E_8$ and that
$$(K,L,q)=(K,L_1,q_1)\perp(K,L_2,q_2)$$
with $q_2$ of type $E_7$, $q_1$ similar to the reduced norm of the
quaternion division algebra representing ${\rm clif}(q_2)$ and $1\in
L_2$. Then the following hold:
\begin{thmenum}
\item There exists $e\in X_{\rm sep}$ such that $\theta(e,L_i)=L_i$
  for $i=1$ and $2$.
\item Let $e\in X$ be as in (i), let $X_e$ be the subspace of $X$
  generated by elements of the form $ev_1v_2\cdots v_j$, where $v_i\in
  L_2$ for $i\in[1,j]$ and $j\ge1$ is arbitrary, let $\cdot_e$, $h_e$,
  respectively, $\theta_e$ denote the restriction of $\cdot$, $h$,
  respectively, $\theta$ to $X_e\times L_2$, $X_e\times X_e$,
  respectively, $X_e\times L_2$ and let
$$\Xi_e=(K,L_2,q_2,1,X_e,\cdot_e,h_e,\theta_e).$$
Then $\Xi_e$ is a quadrangular algebra.
\end{thmenum}
\end{proposition}

\begin{proof}
  By \ref{abc5a}, \ref{abc88} and Proposition~\ref{abc14a}, both
  $(q_1)_E$ and $(q_2)_E$ are hyperbolic. We can thus choose $V_i$ and
  $q_i$ for $i\in[1,d]$ as in Proposition~\ref{abc14a}(iii) with
  $v_1=1$, $V_i\subset L_2$ for $i\in[1,4]$ and $V_i\subset L_1$ for
  $i\in[5,6]$.  By Proposition~\ref{abc14b}, therefore, there exists
  $e\in X_{\rm sep}$ such that $\theta(e,L_i)=L_i$ for $i=1$ and
  $2$. Thus (i) holds.

  Let $\Xi_e$ be as described in (ii). To show that $\Xi_e$ is a
  quadrangular algebra, it therefore suffices to show that $X_e\cdot
  L_2\subset X_e$, $\theta(X_e,L_2)\subset L_2$ and $h(X_e,X_e)\subset
  L_2$.  The first of these inclusions holds by the definition of
  $X_e$. To show the other two inclusions, we first choose non-zero
  elements $v_i\in V_i$ for $i\in[2,5]$. We can assume that $e$ is the
  element of $X$ chosen in \cite[6.4]{weiss-quad}.  Thus the set
  $1,v_2,\ldots,v_5$ is $e$-orthogonal as defined in
  \cite[6.6]{weiss-quad}.  By \cite[1.17(A3) and
  Prop.~6.16]{weiss-quad}, there exists a non-zero $v_6\in L$ such
  that $1,v_2,\ldots,v_5,v_6$ is $e$-orthogonal and
  $ev_2v_3v_4v_5v_6=e$. (We are not claiming that $v_6\in V_6$ or even
  $v_6\in L_1$.)  Let $I_2$ be as in \cite[6.32]{weiss-quad}, let $J$
  denote subset of $I_2$ containing all the elements of $I_2$ that are
  subsets of $\{v_2,v_3,v_4\}$ together with the element
  $\{v_5,v_6\}\in I_2$ (so $|J|=8$), let $J_2$ be the elements of $J$
  of cardinality~2 (so $|J_2|=4$), let $X_x$ for each $x\in J$ be as
  in \cite[6.35]{weiss-quad}, let $M$ be the subspace of $X$ spanned
  by $\{X_m\mid m\in J\}$ and let $N$ be the subspace of $M$ spanned
  by $\{X_m\mid m\in J_2\}$.  By \cite[Prop.~6.34]{weiss-quad},
  $\dim_KM=16$ and $M=eL_2\oplus N$.  By \cite[6.37]{weiss-quad}, we
  have $M=X_e$, by \cite[Prop.~6.13]{weiss-quad}, we have $h(e,N)=0$
  and by \cite[Props.~3.15 and~4.5(i)]{weiss-quad}, $h(e,eL_2)\subset
  L_2$ (since $\theta(e,L_2)\subset L_2$). Hence $h(e,X_e)\subset
  L_2$. By repeated application of \cite[1.17(B1)--(B2)]{weiss-quad},
  it follows that $h(X_e,X_e)\subset L_2$. Since $1\in L_2$, we have
  $L_2^\sigma\subset L_2$, where $\sigma$ is as in \ref{abc7x}.  By
  repeated application of \cite[1.17(C3)--(C4)]{weiss-quad}, it
  follows from $\theta(e,L_2)\subset L_2$ first that
  $\theta(eL_2,L_2)\subset L_2$ and then that $\theta(X_e,L_2)\subset
  L_2$. Thus (ii) holds.
\end{proof}

\begin{definition}\label{abc98a}
  For each non-zero $u$ in $L$, let $\pi_u$ be the reflection of $q$
  given by
$$\pi_u(v)=f(u,v)u/q(u)-v$$
for all $v\in L$. Thus $\pi_1=\sigma$, where $\sigma$ is as in
\ref{abc7x}.
\end{definition}

\begin{proposition}\label{abc98}
  Let $H^\dagger$ and $x_4$ be as in \ref{abc10}.  Suppose that
  $\varphi$ is a product of an even number of reflections of $q$ as
  defined in \ref{abc98a}. Then there exists an element $h\in
  H^\dagger$ such that
$$x_4(v)^h=x_4(\varphi(v))$$
for all $v$.
\end{proposition}

\begin{proof}
  Let $u$ be a non-zero element of $L$.  By
  \cite[eq. (6)--(14)]{weiss-crelle}, there are elements $w_1(0,q(u))$
  and $w_4(u)$ in $H^\dagger$ such that
$$x_4(v)^{w_1(0,q(u))w_4(u)}=x_4(v/q(u))^{w_4(u)}=x_4(\pi_u\pi_1(v))$$
for each $v\in L$.
\end{proof}

\begin{notation}\label{abc99x}
  Let $M$ denote the subgroup of $K^\times$ generated by the non-zero
  elements in the set $\{q(\pi(a)+t)\mid (a,t)\in X\times K\}$.
\end{notation}

Thus
\begin{equation}\label{abc99y}
  M\subset G(q)\cap{\rm Hyp}_2(q)
\end{equation}
by Proposition~\ref{abc12} and $K^{\times2}=\{q(\pi(a)+t)\mid (a,t)\in
\{0\}\times K^\times\}\subset M$.

\begin{proposition}\label{abc99}
  Let $H^\dagger$ and $x_4$ be as in \ref{abc10}.  For each $h\in
  H^\dagger$, there exists a unique similitude $\varphi_h$ of $q$ such
  that
$$x_4(v)^h=x_4(\varphi_h(v))$$
for all $v\in L$. Furthermore, the map $h\mapsto\gamma_h$ is a
surjective homomorphism from $H^\dagger$ to $M$, where $\gamma_h$ is
the multiplier of $\varphi_h$.
\end{proposition}

\begin{proof}
  Let $w_1(a,t)$ for non-zero $(a,t)\in X\times K$ and $w_4(u)$ for
  non-zero $u\in L$ be as in \cite[eqs. (6)--(7)]{weiss-crelle}. Then
$$x_4(v)^{w_4(u)}=x_4\big(uf(u,v^\sigma)-q(u)v^\sigma\big)$$
and
$$x_4(v)^{w_4(a,t)}=x_4\big((\theta(a,v)+tv)/q(\pi(a)+t)\big)$$
for all $v\in L$ and all non-zero $(a,t)\in X\times K$ by
\cite[eqs. (13)--(14)]{weiss-crelle}.  We have
$$q\big((\theta(a,v)+tv)/q(\pi(a)+t)\big)=q(v)/q(\pi(a)+t)$$
for all $v\in L$ and all non-zero $(a,t)\in X\times K$ (by
\ref{abc12}) and
$$q\big(uf(u,v^\sigma)-q(u)v^\sigma\big)=q(v)q(u)^2$$
for all $u,v\in L$ since $q(v^\sigma)=q(v)$.  The claim holds,
therefore, by \cite[Thm.~2.1]{weiss-crelle}.
\end{proof}

\begin{proposition}\label{abc15}
  If $(K,L,q)$ is of type $E_7$, then $G(q)=M$.
\end{proposition}

\begin{proof}
  By \eqref{abc99y}, it suffices to show that $G(q)\subset M$.  Let
  $\varphi_1$ be a similitude of $q$, let $x_4$, $U_4$, $H_0$ and
  $H^\dagger\subset H_0$ be as in \ref{abc10} and let $h$ and
  $\varphi$ be as in Proposition~\ref{abc11}. Thus
$$x_4(v)^h=x_4(\varphi(v))$$
for each $v\in L$ and $\varphi$ is a similitude of $q$ with the same
multiplier as $\varphi_1$.  Let $H_1$ and $H_2$ be the subgroups of
$H_0$ defined in \cite[3.12 and 3.14]{weiss-crelle}.  By
\cite[Thm.~3.15(ii)]{weiss-crelle}, $H_0=H_1H_2$ and by
\cite[Thm.~5.19]{weiss-crelle}, $H_2\subset H_1H^\dagger$. We conclude
that $H_0=H_1H^\dagger$.  By \cite[Prop.~3.11]{weiss-crelle}, $H_1$
centralizes $U_4$.  There thus exists $g\in H^\dagger$ such
$x_4(v)^g=x_4(\varphi(v))$ for each $v\in L$. The claim holds,
therefore, by Proposition~\ref{abc99}.
\end{proof}

\begin{proposition}\label{abc16}
  If $(K,L,q)$ is of type $E_8$, then $G(q)=M$.
\end{proposition}

\begin{proof}
  By \eqref{abc99y}, it suffices to show that $G(q)\subset M$. Let
  $\varphi$ be a similitude of $q$ whose multiplier is not in
  $K^{\times2}$. By Proposition~\ref{abc13}, it suffices to assume
  that $\varphi$ is separable. Let
$$(K,L,q)=(K,L_1,q_1)\perp(K,L_2,q_2)$$
be the decomposition of $q$ obtained by applying
Proposition~\ref{pdq1} to $\varphi$.  Replacing $\Xi$ by an isotope as
defined in \cite[8.7]{weiss-quad}, we can assume that the base point
$1$ lies in $L_2$ (without changing the subgroup generated by the set
of non-zero elements in $\{q(\pi(a)+t)\mid (a,t)\in X\times K\}$). We
can thus let $e$ and
$$\Xi_e=(K,L_2,q_2,1,X_e,\cdot_e,h_e,\theta_e)$$
with $X_e\subset X$ be as in Proposition~\ref{abc14}. By
Proposition~\ref{abc15} (and the uniqueness assertion in
Proposition~\ref{abc7}), we conclude that $\gamma$ is the product of
elements in $\{q(\pi(a)+t)\mid (a,t)\in X_e\times K\}$.
\end{proof}

\smallskip We can now prove Theorems~\ref{abc2} and~\ref{abc3}. By
Proposition~\ref{abc6}, a quadratic form satisfying the hypotheses of
Theorem~\ref{abc2} is of type $E_7$ or $E_8$.  By the existence
assertion in Proposition~\ref{abc7}, we can apply all the results in
this section.  Hence $G(q)\subset{\rm Hyp}_2(q)$ by \eqref{abc99y} and
Propositions~\ref{abc15} and~\ref{abc16}.  By
\cite[Thm. 20.14]{elman}, we have ${\rm Hyp}_2(q)\subset G(q)$.  This
concludes the proof of Theorem~\ref{abc2}.

\smallskip Suppose that $(K,L,q)$ is of type $E_8$ and that $x_4$,
$H_0$ and $H^\dagger$ are as in \ref{abc10}.  To prove
Theorem~\ref{abc3}, it suffices by Proposition~\ref{abc10a} (and the
existence assertion in Proposition~\ref{abc7}) to show that every
element in $H_0$ lies in $H^\dagger$. Let $h\in H_0$. By
\cite[eq. (19)]{weiss-crelle}, there is a similitude $\varphi$ of $q$
such that
$$x_4(v)^h=x_4(\varphi(v))$$
for all $v\in L$. Replacing $h$ by a suitable element in $hH^\dagger$,
we can assume, by Propositions~\ref{abc99} and~\ref{abc16}, that
$\varphi$ is an isometry of $q$ and hence a product of reflections of
$q$. Again replacing $h$ by a suitable element of $hH^\dagger$, we can
assume, by Proposition~\ref{abc98} and
\cite[Prop.~3.16]{weiss-crelle}, that $\varphi$ is the identity. By
\cite[Thm.~3.12]{weiss-crelle}, $h=\alpha_u$ for some $u\in C^\times$,
where $C=K$ by \cite[3.6]{weiss-crelle} and $\alpha_u$ is as defined
in \cite[Prop.~3.11]{weiss-crelle}.  By
\cite[Prop.~3.13]{weiss-crelle}, it follows that $h\in H^\dagger$.
This concludes the proof of Theorem~\ref{abc3}.

\smallskip
\section{$R$-equivalence and an alternative proof of
  Theorem~\ref{abc3}}\label{abc96}

\noindent
In this section, we give an alternative proof of Theorem~\ref{abc3}
based on Corollary~\ref{pdq2} and various other results about
$R$-equivalence. This proof is due to Skip Garibaldi. The methods
employed in this section are completely different from those employed
in the previous section; in particular, we make no further reference
to the Moufang quadrangle $\Gamma$ of \S\ref{abc4}.  \smallskip\par
Let $G$ denote a reductive algebraic group of absolute type $E_8$
whose Tits index over a field $K$ is $E_{8,2}^{66}$. Our goal is to
show that the group of $K$-rational points of $G$ is generated by its
root groups. By \cite[7.3]{gille}, it suffices to assume that ${\rm
  char}(K)=0$.  This will allow us to apply Corollary~\ref{pdq2}.  By
\cite[7.2]{gille}, it suffices to show that $G$ is $R$-trivial.

Now fix a maximal $K$-torus $T$ containing a maximal $K$-split torus
$S$ in $G$ and fix a pinning for $G$ with respect to $T$ over an
algebraic closure of $K$.  Number the simple roots $\alpha_j$ as in
\cite[Chapter~6, Plate~VII]{bourbaki} and let $\omega_i^\vee$ be the
corresponding fundamental dominant co-weights, so
$\langle\alpha_j,\omega_i^\vee\rangle=\delta_{ij}$.  The fundamental
co-weights $\omega^\vee_1$ and $\omega^\vee_8$ belong to the co-root
lattice and so define cocharacters, in other words, homomorphisms from
${\mathbf G}_m$ to $T$. Their images generate a subtorus $S$ in $T$
which is the connected component of the intersection (in $T$) of the
kernels of the roots $\alpha_2,\ldots,\alpha_7$. There is a canonical
isomorphism
\begin{equation}\label{pdq10}
  \quad \Phi:\bar{K}^\times \otimes_{\mathbb{Z}} T_*
  \xrightarrow{\sim} T(\bar K),
\end{equation}
where $T_*$ is the lattice of cocharacters of $T$ and where $\bar{K}$
is an algebraic closure of $K$; see \cite[3.2.11]{springer}.  Since
the group $G$ is of adjoint type, the $\omega_i^\vee$ form a
$\mathbb{Z}$-basis for $T_*$.  Thus we may view $\Phi$ as an
isomorphism
$$\prod_{i=1}^8 \bar{K}^\times \otimes_{\mathbb{Z}} \mathbb{Z}
\omega_i^\vee \xrightarrow{\sim} T(\bar{K}).$$ This shows that the
intersection of the kernels of the roots $\alpha_2\,\ldots,\alpha_7$
is connected and that $\Phi$ restricts to an isomorphism
\begin{equation}\label{pdq12}
  \quad \left ( \bar{K}^\times \otimes_{\mathbb{Z}}
    \mathbb{Z}\omega_1^\vee \right ) \times \left (
    \bar{K}^\times \otimes_{\mathbb{Z}}
    \mathbb{Z}\omega_8^\vee \right ) \xrightarrow{\sim} S(\bar{K}).
\end{equation}
The cocharacters $\omega_1^\vee$ and $\omega_8^\vee$ are defined over
$K$ by \cite[Cor. 6.9]{borel-tits}, so \eqref{pdq12} implies that $S$
is $K$-isomorphic to the direct product of the images of the
cocharacters $\omega_1^\vee$ and $\omega_8^\vee$.

We next fix a parabolic $P$ of $G$ whose Levi subgroup is the
connected reductive group $Z_G(S)$; see \cite[\S13.4 and Lemma
15.1.2]{springer}. Let $U$ be the unipotent radical of $P$ and let
$U^-$ be the unipotent radical of the opposite parabolic. The product
$U^-\times U$ is isomorphic as a variety to an affine space. By
\cite[Proof of Thm.~21.20]{borel}, the natural map from $G$ to $G/P$
restricts to an isomorphism from $U^-$ to an open subset of $G/P$.
Hence the product map from $U^-\times P$ to $G$ defines an isomorphism
from $U^-\times P$ to an open subset of $G$. It follows that $G$ is
birationally equivalent to
$$U^-\times Z_G(S)\times U.$$ 
(This subvariety is the analog of the big cell for the Bruhat
decomposition of $G$ over $K$; see \cite[Prop. 4.10(d)]{borel-tits}.)
We conclude that $G$ is birationally equivalent to the product of
$Z_G(S)$ and an affine space.

Let $H$ denote the derived subgroup of $Z_G(S)$. The sequence
$$1 \to S \to Z_G(S) \to H / (H \cap S) \to 1$$
is exact on $L$-points for every extension $L/K$ because $S$ is split.
Hence $Z_G(S)$ is birationally equivalent to the product of $S$ with
$H / (H \cap S)$.

The absolute Dynkin diagram of $H$ is of type $D_6$.  By
\cite[p.\,211]{tits-crelle}, the group $H$ is ${\rm Spin}(q)$ for $q$
a quadratic form over $K$ with $\dim q=12$, $\disc q=1$ and ${\rm
  clif}(q)$ split.  As $S$ centralizes $H$, the intersection $H\cap S$
is contained in the center $\mu_2\times\mu_2$ of ${\rm Spin}(q)$.  We
show that $H\cap S$ is equal to the center of ${\rm Spin}(q)$.

Since $G$ is simply connected as well as adjoint, the co-roots
$\alpha_j^\vee$ provide also a $\mathbb{Z}$-basis for the cocharacter
lattice $T_*$.  Thus we may view the isomorphism $\Phi$ in
\eqref{pdq10} as an isomorphism
\begin{equation*}
  \bar K^\times \otimes_{\mathbb{Z}} T_*
  = \prod_{i=1}^8 \bar K^\times \otimes_{\mathbb{Z}}
  \mathbb{Z} \alpha_i^\vee \xrightarrow{\sim} T(\bar K).
\end{equation*}
Then it follows from \cite[8.1.8]{springer} that $\Phi$ restricts to
an isomorphism
\begin{equation}\label{pdq13}
  \prod_{i=2}^7 \bar K^\times \otimes_{\mathbb{Z}}
  \mathbb{Z} \alpha_i^\vee \xrightarrow{\sim} (H \cap T)^0(\bar K).
\end{equation}
The expressions for the fundamental dominant weights $\omega_i$ in
terms of the roots $\alpha_j$ in \cite[Chapter~6, Plate~VII]{bourbaki}
imply expressions for the fundamental dominant co-weights
$\omega_i^\vee$ in terms of the co-roots $\alpha_j^\vee$. These
expressions yield
$$\omega_1^\vee(-1) = \alpha_2^\vee(-1) \alpha_3^\vee(-1)\quad\text{and}\quad
\omega_8^\vee(-1) = \alpha_3^\vee(-1) \alpha_5^\vee(-1)
\alpha_7^\vee(-1).$$ From \eqref{pdq12} and \eqref{pdq13}, we see that
these two elements both lie in $S(\bar K)$ and in $(H\cap
T)^\circ(\bar K)$ and are nontrivial and distinct.  We have thus
produced two distinct nontrivial elements in $(H\cap S)(\bar K)$.
Hence $H\cap S$ is, in fact, the entire center of ${\rm Spin}(q)$.

Therefore $H/(H\cap S)$ is $\PGO_+(q)$. It follows by \ref{pdq2} that
$H/(H\cap S)$ is $R$-trivial. Therefore $G$ is birationally equivalent
to the product of $\PGO_+(q)$ times an affine space. Thus $G$ itself
is $R$-trivial by \cite[p. 197, Cor.]{sansuc}. This concludes our
second proof of Theorem~\ref{abc3}.

\smallskip We observe that this proof goes through verbatim for every
group $G$ of absolute type $E_8$ in whose Tits index the roots
$\alpha_1$ and $\alpha_8$ are circled.  We conclude that for such
groups, $G$ is $R$-trivial and the group of $K$-rational points of $G$
is generated by its root groups.  With only minor modifications, the
proof also shows that if $G$ is adjoint of absolute type $E_7$ with
trivial Tits algebras and the root $\alpha_1$ is circled in the Tits
index of $G$, then $G$ is $R$-trivial.

\smallskip
\section{Theorem~\ref{abc2} and triality}
\label{sec:trial}
\noindent
In this section, we assume that ${\rm char}(K)\ne2$. We give an
alternative proof of Theorem~\ref{abc2}, based on 
completely different methods. We actually show:

\begin{proposition}
  \label{prop:Hyp2}
  Suppose the characteristic of the base field $K$ is different
  from~$2$.  If $q$ is a quadratic form with trivial discriminant,
  then $\Hyp_2(q)=\Hyp(q)$ in the following cases:
  \begin{thmenum}
  \item $\dim q=8$ and the index of ${\rm clif}(q)$ is $1$ or $2$;
  \item $\dim q=12$ and ${\rm clif}(q)$ is split.
  \end{thmenum}
\end{proposition}

\noindent
Since in each case $G(q)=\Hyp(q)$ by Proposition~\ref{prop:E7G=H} and
Corollary~\ref{cor:E8G=H}, Theorem~\ref{abc2} follows from
Proposition~\ref{prop:Hyp2}.  \medbreak\par
We start with the case of $8$-dimensional quadratic forms. If ${\rm
  clif}(q)$ is split, then $q$ is a multiple of a $3$-fold Pfister
form, and the result follows from Lemma~\ref{lem:pfister}. Similarly,
if $q$ is isotropic, then $q$ is Witt-equivalent to a multiple of a
$2$-fold Pfister form, and the result follows from
Lemma~\ref{lem:pfister}. We may thus assume that $(K,L,q)$ is of type
$E_7$ and let $D$ be the quaternion division algebra over $K$ that
represents ${\rm clif}(q)$. We show next that the Clifford algebra
construction associates to $q$ a skew-hermitian form $h$ of rank~$4$
over $D$, and we shall complete the proof of
Proposition~\ref{prop:Hyp2}(i) by proving that
\[
\Hyp(q)=\Sn(h)=\Hyp_2(q);
\]
see Proposition~\ref{prop:HypSn}.
\medbreak\par
Let $(A,\sigma)$ be a central simple $K$-algebra of degree~$8$ with an
orthogonal involution of trivial discriminant. The Clifford algebra
$C(A,\sigma)$ decomposes into a direct product of two central simple
$K$-algebras of degree~$8$:
\[
C(A,\sigma)=C_+(A,\sigma)\times C_-(A,\sigma).
\]
Recall that $C(A,\sigma)$ carries a canonical involution
$\underline{\sigma}$, which induces orthogonal involutions $\sigma_+$
and $\sigma_-$ on $C_+(A,\sigma)$ and $C_-(A,\sigma)$ respectively. By
triality (see \cite[(42.3)]{BoI}), the Clifford algebras of
$(C_+(A,\sigma),\sigma_+)$ and $(C_-(A,\sigma),\sigma_-)$ satisfy
\begin{align*}
  (C(C_+(A,\sigma),\sigma_+),\underline{\sigma_+}) & =
  (C_-(A,\sigma),\sigma_-)\times (A,\sigma),\\
  (C(C_-(A,\sigma),\sigma_-),\underline{\sigma_-}) & = (A,\sigma)
  \times (C_+(A,\sigma),\sigma_+).
\end{align*}

\begin{proposition}
  \label{trial:prop}
  The following hold:
  \begin{enumerate}
  \item[(1)] If $A$ is split, then $(C_+(A,\sigma),\sigma_+)$ and
    $(C_-(A,\sigma),\sigma_-)$ are isomorphic.
  \item[(2)] If $(A,\sigma)$ is split and isotropic, then
    $(C_+(A,\sigma),\sigma_+)$ and $(C_-(A,\sigma),\sigma_-)$ are
    hyperbolic.
  \item[(3)] If $(A,\sigma)$ is split and hyperbolic, then
    $(C_+(A,\sigma),\sigma_+)$ and $(C_-(A,\sigma),\sigma_-)$ are
    split and hyperbolic.
  \end{enumerate}
\end{proposition}

\begin{proof}
  (1) is well-known, (2) is in \cite[(8.5)]{BoI}, and (3) follows from
  (2) and the fact that the Clifford invariant of a hyperbolic
  quadratic form is trivial.
\end{proof}

We apply this proposition in the following context: let $(K,L,q)$ be
an $8$-dimen\-sional quadratic space with $\disc q=1$, and assume ${\rm
  clif}(q)$ is represented by a quaternion division algebra $D$. Let
$\ad_q\colon\End_KL\to \End_KL$ be the adjoint involution of $q$. We
apply the discussion above with $(A,\sigma)=(\End_KL,\ad_q)$. Then
$C(A,\sigma)=C_0(L,q)$ and $(C_+(A,\sigma),\sigma_+)$,
$(C_-(A,\sigma),\sigma_-)$ are isomorphic to $(\End_DW,\ad_h)$ for
some $4$-dimensional skew-hermitian space $(W,h)$ over $D$ (with its
conjugation involution).

\begin{proposition}
  \label{trial2:prop}
  For an arbitrary extension $E/K$, the following statements are
  equivalent:
  \begin{enumerate}
  \item[(a)] $q_E$ is hyperbolic;
  \item[(b)] $D_E$ is split and $(\End_DW,\ad_h)_E$ is hyperbolic;
  \item[(c)] $D_E$ is split and $(\End_DW,\ad_h)_E$ is isotropic.
  \end{enumerate}
\end{proposition}

\begin{proof}
  (a)~$\Rightarrow$~(b): This readily follows from
  Proposition~\ref{trial:prop}(3).

  (b)~$\Rightarrow$~(c): Clear.

  (c)~$\Rightarrow$~(a): This follows from
  Proposition~\ref{trial:prop}(2) with $(\End_DW,\ad_h)_E$ for
  $(A,\sigma)$; then by triality $(C_+(A,\sigma),\sigma_+)$ or
  $(C_-(A,\sigma),\sigma_-)$ is isomorphic to $(\End_KL,\ad_q)_E$).
\end{proof}

The next results \ref{lem:quadsplit}--\ref{cor:perf} hold for
skew-hermitian forms of arbitrary dimension.

\begin{lemma}
  \label{lem:quadsplit}
  Let $(W,h)$ be a skew-hermitian space over a quaternion division
  algebra $D$ over $K$ and let $E$ be a quadratic extension of $K$. If
  $h$ is anisotropic, the following conditions are equivalent:
  \begin{enumerate}
  \item[(i)] $E\cong K\bigl(h(v,v)\bigr)$ for some $v\in W$;
  \item[(ii)] $D_E$ is split and $h_E$ is isotropic.
  \end{enumerate}
\end{lemma}

\begin{proof}
  If (i) holds, then $E$ is isomorphic to a maximal subfield of $D$,
  hence $D_E$ is split. Let $h(v,v)^2=a\in K^\times$, so $E\cong
  K(\sqrt{a})$. Then $v\cdot(h(v,v)+\sqrt{a})\in W_E$ is isotropic for
  $h_E$. Thus, (ii) holds.

  Conversely, if (ii) holds, then $E$ is isomorphic to a maximal
  subfield of $D$. Let $E=K(\sqrt{a})$ for some $a\in K$, and let
  $\lambda\in D$ be a pure quaternion such that $\lambda^2=a$. Suppose
  $x+y\sqrt{a}\in W_E$ is $h_E$-isotropic for some $x$, $y\in W$. The
  condition $h_E(x+y\sqrt{a}, x+y\sqrt{a})=0$ yields
  \[
  h(x,x)+h(y,y)a=0 \quad\text{and}\quad h(x,y)+h(y,x)=0.
  \]
  Since $h$ is skew-hermitian, the second equation shows that
  $h(x,y)\in K$. Then
  \[
  h(x+y\lambda,x+y\lambda) = 2h(x,y)\lambda - h(y,y)a -\lambda h(y,y)
  \lambda
  \]
  and the right side commutes with $\lambda$. Therefore,
  $h(x+y\lambda, x+y\lambda)=\lambda b$ for some $b\in K^\times$, and
  we have $E\cong K\bigl(h(v,v)\bigr)$ with $v=x+y\lambda$.
\end{proof}

For any skew-hermitian space $(W,h)$ over a quaternion division
algebra $D$ over $K$, we let $\Sn(h)$ denote the group of spinor norms
of $h$, which is the image of the Clifford group
$\Gamma(\End_DW,\ad_h)=\Gamma(W,h)$ under the multiplier map; see
\cite[(13.30)]{BoI}. 

\begin{proposition}
  \label{prop:sn}
  If $h$ is anisotropic, then $\Sn(h)=\prod_E N(E/K)$, where $E$ runs
  over the quadratic extensions of $K$ satisfying the equivalent
  conditions~(i) and (ii) of Lemma~\ref{lem:quadsplit}.
\end{proposition}

\begin{proof}
  The multiplier map $\Gamma(W,h)\to K^\times$ factors
  through the vector representation $\Gamma(W,h)\to\Orth_+(W,h)$,
  where $\Orth_+(W,h)$ is
  the group of direct isometries of the space $(W,h)$. By
  \cite[Thm.~6.2.17]{HO}, this group is generated by transformations
  of the form
  \[
  \tau_{v,r}\colon W\to W,\quad x\mapsto x-vh(vr,x)
  \]
  where $v\in W$ is an anisotropic vector and $r\in D^\times$
  satisfies $r-\overline{r}=rh(v,v)\overline{r}$. To compute the
  spinor norm of that transformation, observe that $\tau_{v,r}$ is the
  identity on $v^\perp$, hence the spinor norm of $\tau_{v,r}$ is the
  spinor norm of its restriction to the $1$-dimensional subspace
  $vD$. Let $\nu=h(v,v)\in D^\times$ and let $h_v$ denote the
  restriction of $h$ to $vD$, so
  \[
  h_v(v\lambda,v\mu)=\overline{\lambda}\nu\mu \qquad\text{for
    $\lambda$, $\mu\in D$.}
  \]
  We have
  \[
  \Orth_+(vD,h_v) = \{\theta\in K(\nu)^\times\mid
  \theta\overline{\theta}=1\} \quad\text{and}\quad \Gamma(vD,h_v) =
  K(\nu)^\times
  \]
  (where $\theta\in K(\nu)^\times$ is identified with the map
  $v\lambda\mapsto v\theta\lambda$ for $\lambda\in D$). The vector
  representation $\Gamma(vD,h_v)\to\Orth_+(vD,h_v)$ carries $u\in
  K(\nu)^\times$ to $u\overline{u}^{-1}$, hence the spinor norm of
  that isometry is $u\overline{u}K^{\times2}$; see
  \cite[(13.17)]{BoI}. This shows that $\Sn(h_v)$ consists of norms
  from the quadratic extension $K(\nu)/K$. Since $\Sn(h)$ is generated
  by the groups $\Sn(h_v)$ for the anisotropic vectors $v\in W$, the
  proposition follows.
\end{proof}

\begin{corollary}
  \label{cor:perf}
  Let $(W,h)$ be a skew-hermitian space over a quaternion division
  algebra $D$ over $K$, and let $p=\charac(K)>2$. For $\widetilde
  K=K^{-p^{-\infty}}$ the perfect closure of $K$, we have
  $\Sn(h_{\widetilde K})\cap K=\Sn(h)$.
\end{corollary}

\begin{proof}
  The inclusion $\Sn(h)\subset \Sn(h_{\widetilde K})\cap K$ is clear,
  so it suffices to prove the reverse inclusion. Let $x\in
  \Sn(h_{\widetilde K})\cap K$. If $x\in\widetilde K^{\times2}$, then
  $x\in K^{\times2}\subset\Sn(h)$. We may thus assume $x\notin
  \widetilde K^{\times2}$. By Proposition~\ref{prop:sn}, there exist
  quadratic extensions $\widetilde E_1/\widetilde K$, \ldots,
  $\widetilde E_r/\widetilde K$ such that $D_{\widetilde E_i}$ is
  split and $h_{\widetilde E_i}$ is isotropic for each $i\in[1,r]$,
  and elements $y_i\in \widetilde E_i\setminus \widetilde K$ for
  $i\in[1,r]$ such that
  \begin{equation}
    \label{eq:perf1}
    x=N_{\widetilde E_1/\widetilde K}(y_1)\cdot\ldots\cdot
    N_{\widetilde E_r/\widetilde K}(y_r).
  \end{equation}
  Let $K'\subset\widetilde K$ be the subfield generated by
  $N_{\widetilde E_1/\widetilde K}(y_1)$, \ldots, $N_{\widetilde
    E_r/\widetilde K}(y_r)$ and, for $i\in[1,r]$, let
  $E'_i=K'(y_i)$. Thus, $K'/K$ is a purely inseparable extension of
  finite degree, \eqref{eq:perf1} yields
  \begin{equation}
    \label{eq:perf2}
    x=N_{E'_1/K'}(y_1)\cdot\ldots\cdot N_{E'_r/K'}(y_r),
  \end{equation}
  and each $E'_i/K'$ is a quadratic extension. For $i\in[1,r]$,
  let $E_i''$ be the separable closure of $K$ in $E'_i$; it is a
  quadratic extension of $K$ and we have 
  \[
  E'_i\cong E_i''\otimes_KK'\qquad\text{and}\qquad
  \widetilde E_i\cong E_i''\otimes_K\widetilde K.
  \]
  Since
  quaternion division algebras do not split over extensions of odd
  degree, the condition that $D_{\widetilde E_i}$ is split shows that
  $D_{E''_i}$ is split. Likewise, anisotropic skew-hermitian forms do
  not become isotropic over odd-degree extensions by
  \cite[Thm.~3.5]{PSS}, hence $h_{E''_i}$ is isotropic. Now, let
  $[K':K]=p^d$; taking the norm from $K'$ to $K$ of each side of
  \eqref{eq:perf2}, we obtain
  \[
  x^{p^d} = N_{E'_1/K}(y_1)\cdot\ldots\cdot N_{E'_r/K}(y_r) =
  N_{E''_1/K}\bigl(N_{E'_1/E''_1}(y_1)\bigr)\cdot\ldots\cdot
  N_{E''_r/K}\bigl(N_{E'_r/E''_r}(y_r)\bigr).
  \]
  Since $x^{p^d}\equiv x\bmod K^{\times2}$, this equation shows that
  $x$ is a product of norms from quadratic extensions over which $D$
  is split and $h$ is isotropic, hence $x\in\Sn(h)$ by
  Proposition~\ref{prop:sn}. 
\end{proof}

We now return to the context of Proposition~\ref{trial2:prop}. The
following proposition completes the proof of
Proposition~\ref{prop:Hyp2}(i):

\begin{proposition}
  \label{prop:HypSn}
  For $q$ an anisotropic $8$-dimensional quadratic form and $h$ the
  corresponding 
  $4$-dimensional skew-hermitian form as in
  Proposition~\ref{trial2:prop}, we have
  \[
  \Hyp(q)=\Sn(h)=\Hyp_2(q).
  \]
\end{proposition}

\begin{proof}
  Proposition~\ref{trial2:prop} shows that the quadratic extensions
  $E/K$ such that $D_E$ is split and $h_E$ is isotropic are exactly
  those such that $q_E$ is hyperbolic, hence by
  Proposition~\ref{prop:sn} we have
  \[
  \Sn(h)=\Hyp_2(q)\subset\Hyp(q).
  \]
  To complete the proof, we show $\Hyp(q)\subset\Sn(h)$. Let $E/K$ be
  a finite-degree extension such that $q_E$ is hyperbolic, let
  $\widetilde K$ be the perfect closure of $K$ in some algebraic
  closure of $E$, and let $K_1$ be the purely inseparable closure of
  $K$ in $E$. The compositum $E\cdot\widetilde K$ of $E$ and
  $\widetilde K$ satisfies $E\cdot\widetilde K\cong
  E\otimes_{K_1}\widetilde K$. Since $q_E$ is hyperbolic, $q$ is also
  hyperbolic over $E\cdot\widetilde K$, hence $D_{E\cdot \widetilde
    K}$ is split and $h_{E\cdot\widetilde K}$ is isotropic, by
  Proposition~\ref{trial2:prop}. Therefore, $\Sn(h_{E\cdot\widetilde
    K})=(E\cdot\widetilde K)^\times$. Since $\widetilde K$ is perfect,
  we may apply the norm principle for spinor norms (see
  \cite[(6.2)]{Merknorm}), which is a twisted analogue of Knebusch's
  norm theorem, to see that $N(E\cdot\widetilde K/\widetilde K)\subset
  \Sn(h_{\widetilde K})$. Since $N(E/K_1)\subset N(E\cdot\widetilde
  K/\widetilde K)$, it follows that $N(E/K_1)\subset \Sn(h_{\widetilde
    K})$. Let $p=\charac(K)$ if $\charac(K)>2$ and $p=1$ if
  $\charac(K)=0$, so $[K_1:K]=p^d$ for some $d\geq0$. For all $x\in
  K_1$, we have $N_{K_1/K}(x)=x^{p^d}$, hence
  \[
  N(E/K)
  =N_{K_1/K}\bigl(N(E/K_1)\bigr)=N(E/K_1)^{p^d}\subset\Sn(h_{\widetilde
    K}).
  \]
  But $N(E/K)\subset K$, hence Corollary~\ref{cor:perf} shows that
  $N(E/K)\subset \Sn(h)$. Of course, we also have
  $K^{\times2}\subset\Sn(h)$, hence $\Hyp(q)\subset\Sn(h)$.
\end{proof}

Part~(ii) of Proposition~\ref{prop:Hyp2} follows from part~(i) by the
same arguments as in the proof of Corollary~\ref{cor:E8G=H}: let $q$
be a $12$-dimensional nondegenerate form with trivial discriminant and
Clifford invariant. If $q$ is isotropic, then it is Witt-equivalent to
a $3$-fold Pfister form and $G(q)=\Hyp_2(q)$ by
Lemma~\ref{lem:pfister}. For the rest of the proof, suppose $q$ is
anisotropic, i.e., $q$ is of type~$E_8$. Let $\gamma\in G(q)$.  Since
${\rm char}(K)\ne2$, all similitudes of $q$ are separable.  We can
thus fix a decomposition $q=q_1\perp q_2$ as in
Proposition~\ref{pdq1}. Since $\gamma\in G(q_2)$, part~(i) of
Proposition~\ref{prop:Hyp2} shows that $\gamma\in \Hyp_2(q_2)$. Since
$\Hyp_2(q_2)\subset\Hyp_2(q)$ by Remark~\ref{rem:E8G=H}, it follows
that $\gamma\in\Hyp_2(q)$. This proves
Proposition~\ref{prop:Hyp2}(ii).

\end{document}